\documentclass[letterpaper, 10 pt, conference]{ieeeconf}  % Comment this line out if you need a4paper

\IEEEoverridecommandlockouts                              % This command is only needed if 
\usepackage{preamble_papers}

\title{\LARGE \bf
Scenario Reduction with Guarantees for Stochastic Optimal Control of Linear Systems
}

\author{Francesco Cordiano, Bart De Schutter
\thanks{
This project has received funding from the European Research Council (ERC) under the European Union's Horizon 2020 research and innovation programme (Grant agreement No. 101018826 - ERC Advanced Grant CLariNet).}
\thanks{
The authors are with the Delft Center for Systems and Control, Delft University of Technology,
Delft, The Netherlands,
{\tt\small \{f.cordiano, b.deschutter\}@tudelft.nl}}%
}

\usepackage{algorithm}
\usepackage{algpseudocode}
\renewcommand{\P}{\mathbb{P}}
\newcommand{\tP}{\Tilde{\mathbb{P}}}
\renewcommand{\E}{\mathbf{E}}
\renewcommand{\Pr}{\mathrm{Pr}}

\begin{document}

\maketitle
\thispagestyle{empty}
\pagestyle{empty}

%%%%%%%%%%%%%%%%%%%%%%%%%%%%%%%%%%%%%%%%%%%%%%%%%%%%%%%%%%%%%%%%%%%%%%%%%%%%%%%%
\begin{abstract}
Scenario reduction algorithms can be an effective means to provide a tractable description of the uncertainty in optimal control problems. However, they might significantly compromise the performance of the controlled system.
In this paper, we propose a method to compensate for the effect of scenario reduction on stochastic optimal control problems for chance-constrained linear systems with additive uncertainty. We consider a setting in which the uncertainty has a discrete distribution, where the number of possible realizations is large. We then propose a reduction algorithm with a problem-dependent loss function, and we define sufficient conditions on the stochastic optimal control problem to ensure out-of-sample guarantees (i.e., against the original distribution of the uncertainty) for the controlled system in terms of performance and chance constraint satisfaction. Finally, we demonstrate the effectiveness of the approach on a numerical example.

\end{abstract}

%%%%%%%%%%%%%%%%%%%%%%%%%%%%%%%%%%%%%%%%%%%%%%%%%%%%%%%%%%%%%%%%%%%%%%%%%%%%%%%%

\section{INTRODUCTION}

Dealing with uncertainty is one of the major challenges in optimal control and decision making. Indeed, several real-world applications involve some level of uncertainty, which can negatively affect the system performance in terms of cost, safety, and reliability. Stochastic optimal control \cite{mesbah2016stochastic} has received significant attention in the recent years due to the possibility of explicitly encoding a probabilistic description of the uncertainty in the optimal control problem.

However, one of the main challenges in stochastic optimal control problems is to provide a description of the uncertainty such that: a) it is accurate enough so that performance guarantees can be given for the controlled system; and b) the resulting optimal control problem is numerically tractable and amenable to online optimization. In this regard, a popular technique to represent the uncertainty in a tractable way consists in employing a set of \emph{scenarios}.
Scenario-based methods are rooted in multistage stochastic programs \cite{ruszczynski2003stochastic}, where a scenario tree is built over the prediction horizon according to the possible realization of the uncertainty \cite{bernardini2012stabilizing}. More generally, scenarios can be intended as time series with an associated probability, which allows the representation of complex stochastic phenomena, e.g., the price of an asset in trading problems \cite{bemporad2010scenariobased, sharma2017scenario}, the solar irradiance and the ambient temperature in building heating systems \cite{pippia2021scenariobased}, or the degree of deterioration of railways in preventive maintenance \cite{su2017multilevel}. Scenario-based methods rely on an atomic description of the uncertainty, which is either assumed to be given (e.g., obtained from a forecast \cite{morales2009scenario}), or estimated via scenario generation methods \cite{sharma2017scenario, pippia2021scenariobased}, or obtained from a continuous distribution via discretization \cite{pflug2001scenario}. 

However, one of the main shortcomings of scenario-based methods is that a large number of scenarios is required to obtain an accurate representation of the uncertainty, especially in the case of a continuous distribution, which is typically approximated by means of a sample-average approximation \cite{pagnoncelli2009sample}. In general, to tackle the issue of the growing complexity, several scenario reduction algorithms have been proposed in the context of stochastic programming, where the aim is to determine a set of scenarios with reduced cardinality that can well approximate some characteristics of the original scenario set. Examples of reduction algorithms can be found in \cite{dupavcova2003scenario, henrion2008discrepancy}, where the accuracy of the approximation is measured in terms of a probability distance (e.g., the Wasserstein distance \cite{romisch2003stability}) between the distributions describing the initial and reduced scenario set, or in \cite{rujeerapaiboon2022scenario, wang2017scenario}, where scenarios are grouped according to their similarity using clustering-based algorithms.

Using a limited number of scenarios is justified by quantitative stability
results for stochastic programs \cite{romisch2007stability, rockafellar2009variational}, where it is proven that the optimal value of a stochastic program, or the set of optimal solutions, possesses some type of continuity properties with respect to changes in the probability distribution of the uncertainty. However, when only a limited number of scenarios is used in a stochastic optimal control problem, guarantees hold only for the \emph{in-sample} performance of the system, i.e.\ the performance against the considered scenario set, whereas the \emph{out-of-sample} (OOS) performance, i.e.\ the performance against the original distribution of the uncertainty, in terms of incurred cost and actual constraint satisfaction, can be seriously compromised. \textcolor{black}{For example, previous works such as
\cite{bernardini2012stabilizing, lucia2012new, holtorf2019multistage}, construct a scenario tree by considering only some possible realizations of the uncertainty, but they do not consider chance constraints, which would require the complete tree structure to provide guarantees for the controlled systems. Chance constraints are considered in \cite{blackmore2010probabilistic} as a sample average approximation, but guarantees hold only if the number of scenarios approaches infinity. Alternatively, they are considered in \cite{prandini2012randomized} via randomized algorithms, where performance guarantees are provided by an appropriate sampling of the constraints, which however might lead to conservative solutions.}

\textcolor{black}{
In light of these considerations,  in this paper we explore the use of scenario reduction algorithms in stochastic optimal control problems for chance-constrained linear systems with additive uncertainty. We consider a setting where the actual distribution of the uncertainty is discrete, and the number of possible realizations (i.e., the scenarios) is finite but huge, which requires a reduction algorithm to allow implementation in an online setting. Our contributions can be summarized in the following points:
\begin{itemize}
    \item We derive a novel problem-dependent clustering-based scenario reduction algorithm that, in contrast to more standard algorithms based on $k$-means \cite{hartigan1979algorithm, rujeerapaiboon2022scenario}, or $k$-medians \cite{whelan2015understanding},  minimizes a loss function that is shown to have a direct connection with the cost function employed in the optimal control problem;
    \item In contrast to other scenario-based methods \cite{ bernardini2012stabilizing, lucia2012new, blackmore2010probabilistic, prandini2012randomized}, we solve a chance-constrained scenario-based optimal control problem by considering only the scenarios resulting from the proposed reduction algorithm, which leads to computational benefits. We then robustify the optimal control problem by taking into account the effect of the discarded scenarios implicitly via constraint tightening, and by over-approximating the cost function.
\end{itemize}
} The resulting optimal control problem benefits from a reduced computational complexity thanks to the usage of a reduced number of scenarios, while preserving OOS guarantees against the original distribution of the scenarios

The rest of the paper is organized as follows. In Section \ref{sec:prob}, we introduce the scenario-based optimal control problem, and in Section \ref{sec:red} we discuss a problem-dependent reduction algorithm based on clustering. In Section \ref{sec:oos} we describe how to modify the constraint set and the cost function in order to preserve OOS guarantees, and in Section \ref{sec:ex} we demonstrate the effectiveness of the proposed approach on a numerical example. Section VI concludes the paper.

\section{Preliminaries}\label{sec:prob}

\subsection{Notation}
The symbol $\mathbf{1}_X$ denotes the indicator function, which takes the value $1$ if the logical statement $X$ is true, and 0 otherwise. With $\|\cdot\|_l$ we denote an $l$-norm, with $l\in\{1,...,\infty\}$. The symbols $\oplus$ and $\ominus$ denote, respectively, the Minkowski sum and difference. The symbols $\Pr$ and $\E$ denote, respectively, the probability and expectation operators. Given $n$ matrices $P_1,...,P_n$, diag$(P_1,...,P_n)$ is the block-diagonal matrix where the diagonal blocks are the matrices $P_1,...,P_n$.

\subsection{Problem statement}
We consider a discrete-time linear system affected by stochastic additive uncertainty:
\begin{align}\label{eq:sys}
    x_{t+1} = Ax_t + Bu_t + \eta_t, 
\end{align}
where $x_t\in\R^n$ and $u_t\in\R^m$ are, respectively, the state and input of the system at time step $t$, and $\eta_t\in\R^n$ is an exogenous additive stochastic signal drawn from a distribution $\P$.
The goal of the controller is to employ the system dynamics \eqref{eq:sys} to find a sequence of inputs $\us_{t},...,\us_{t+N-1}$ such that the expected performance of the system
\begin{align}
    \E\textsubscript{$\eta_{t},...,\eta_{t+N-1}$} \left[\sum_{k=1}^{N} \left(\|Q x_{t+k}\|_1 + \|R u_{t+k-1}\|_1 \right) \right] \label{socp:cost} 
\end{align}
is minimized, while satisfying the following constraints on the state and on the input of the system:
\begin{align}
    & \Pr(x_{t+k}\in\Xcal, \ k=1,...,N) \geq 1-\varepsilon, \label{socp:sc}
    \\ & u_{t+k} \in \Ucal,  \ k=0,...,N-1, \ \label{socp:uc}
\end{align}
\textcolor{black}{where $Q\in\R^{n\times n}$, $R\in\R^{m\times m}$ are matrices with non-negative entries,} $\varepsilon\in[0, 1)$ is a risk-tolerance parameter, and $\Xcal, \Ucal$ are polytopic convex sets, i.e.: $\Xcal:= \{x\in\R^n:H_x x\leq h_x\}$, and $\Ucal:= \{u\in\R^m:H_u u\leq h_u\}$, where $H_x, H_u, h_x, h_u$ are matrices and vectors of appropriate dimensions. \textcolor{black}{For simplicity, in the rest of the paper we assume that $Q$ and $R$ are identity matrices of appropriate dimension, as this does not significantly impact the proposed analysis.}

To derive tractable reformulations for 
\eqref{socp:cost}, \eqref{socp:sc}, a popular choice in the literature is to represent the uncertainty by means of scenarios \cite{bemporad2010scenariobased, pippia2021scenariobased, pagnoncelli2009sample}. In general, the required number of scenarios to obtain a reliable approximation can be large, or even infinite in the case of a continuous distribution \cite{pagnoncelli2009sample}. In this paper, for simplicity, we directly consider the case in which $\P$ is discrete, but the number of the possible realization (i.e., the possible scenarios) can be very large. Let the following assumption hold:
\begin{assumption}\label{ass:eta}
    The stochastic signal $\eta_t, \forall t\geq0,$ is drawn from a discrete distribution $\P$, such that $\Pr(\eta_t = \etaj_t, \forall t\geq0) = \pj, j=1,...,M$, where $M\in\Z_{>0}$ represents the number of possible realizations (or scenarios), and the probabilities $\pj$ satisfy:  $\sum_{j=1}^{M}\pj=1, \pj\in(0,1], j=1,...,M$. Both the scenarios $\etaj$ and the probabilities $\pj$ are assumed to be known.
\end{assumption}

In this setting, a scenario is conceived as a (multidimensional) time series with an associated probability, as proposed, e.g., in \cite{pippia2021scenariobased, morales2009scenario}. This allows to consider time-correlated phenomena (e.g., the solar irradiance in heating control systems \cite{pippia2021scenariobased}), and it also includes the more traditional scenario tree \cite{bernardini2012stabilizing}.
Under Assumption \ref{ass:eta}, a stochastic optimal control problem can be formulated as \cite{pagnoncelli2009sample}
\begin{subequations}\label{sbocp}
\begin{align}
    \min_{u_t,...,u_{t+N-1}} \ & \sum_{j=1}^{M} \pj \sum_{k=1}^{N} \left(\|\xj_{t+k}\|_1 + \|u_{t+k-1}\|_1 \right) \label{sbocp:cost} 
    \\ \textup{s.t.} \ & \xj_t = x_t, \ j=1,...,M,
    \\ & \xj_{t+k+1} = A\xj_{t+k} + Bu_{t+k} + \etaj_{t+k}, \label{sbocp:dyn}
    \\ & \sum_{j=1}^M \pj \mathbf{1}_{\xj_{t+k}\in\Xcal, \ k=1,...,N} \geq 1-\varepsilon, \label{sbocp:sc}
    \\ & u_{t+k} \in \Ucal, \ \label{sbocp:uc}
    \\ & k=0,...,N-1. \nonumber
\end{align}
\end{subequations}
In particular, \eqref{sbocp:cost} is the expectation \eqref{socp:cost} computed over the discrete scenario set, \eqref{sbocp:dyn} describes the evolution \eqref{eq:sys} according to the $j$-th scenario, and \eqref{sbocp:sc} is an equivalent reformulation of \eqref{socp:sc} by means of the indicator function \cite{pagnoncelli2009sample}.
\begin{remark}
    The usage of the 1-norm, even if less common than the squared 2-norm, is frequently used in optimal control, e.g., for piecewise affine systems, where the corresponding optimization problem can be cast as a mixed-integer linear program (MILP) \cite{richards2005mixedinteger}. Notice that also \eqref{sbocp} can be cast as an MILP, as the chance constraint \eqref{sbocp:sc} admits an equivalent reformulation by associating one logical variable to each scenario, whose value equals the indicator function. In addition, the usage of the 1-norm will turn to be useful in the performance analysis in Section \ref{sec:oos}
\end{remark}

In general, one of the main shortcomings in scenario-based formulations of stochastic programs is that the computational complexity of \eqref{sbocp} grows rapidly with the number of scenarios, especially due to the non-smooth chance constraint \eqref{sbocp:sc}. An effective means to reduce the computational burden of \eqref{sbocp} is provided by scenario reduction algorithms \cite{rujeerapaiboon2022scenario}, where from an initial discrete distribution $\P$, the goal is to find a discrete distribution $\tP$ whose domain has a lower cardinality than the domain of $\P$, in a way that $\tP$ provides a good description of $\P$ (e.g., in terms of the Wasserstein distance \cite{romisch2003stability}). Hence, let 
$\teta_t\in\R^n, t\in\Z_{\geq0}$ be a stochastic signal described by the distribution $\tP$, such that: $\Pr(\teta_t = \tetaj_t, \forall t\geq0) = \tpj, \sum_{j=1}^{\tM}=1, \tpj\in(0,1],  j=1,...,\tM$, where $\tM\in\Z_{>0}, \tM<M$, represents the cardinality of the reduced set of scenarios. The way in which the scenarios $\teta^{(1)}, ..., \teta^{(\tM)},$ and the corresponding probabilities $\tp^{(1)},..., \tp^{(\tM)},$ are chosen clearly affects the solution of the resulting optimal control problem; hence, in the following, we provide an equivalent reformulation of \eqref{sbocp}, where the dependence of the cost and the constraint set on the distribution $\tP$ is made explicit.
Let us introduce the following compact notation:
\begin{align}\label{eq:compact_vec}
    & \xxj_t = \begin{bmatrix}
        \xj_{t+1} \\ \vdots \\ \xj_{t+N}
    \end{bmatrix}, \  
    \uu_t = \begin{bmatrix}
        u_{t} \\ \vdots \\ u_{t+N-1}
    \end{bmatrix}, \ 
    \teetaj_t = \begin{bmatrix}
        \tetaj_{t} \\ \vdots \\ \tetaj_{t+N-1}
    \end{bmatrix}, \\
    & j=1,...,\tM, \nonumber
\end{align}
for which we have:
\begin{equation}\label{eq:compact_dyn}
    \xxj_t =   \FF x_t + \GG\uu_t + \GGamma\teetaj_t,
\end{equation}
where $\FF, \GG, \GGamma$ are appropriate matrices derived from the dynamics \eqref{eq:sys}, which can be found e.g.\ in \cite{prandini2012randomized}. 
Let also $\HH_x:=\textup{diag}(H_x, ..., H_x)$,  $\HH_u:=\textup{diag}(H_u, ..., H_u)$, $\hh_x = [h_x^\top,...,h_x^\top]^\top, \hh_u = [h_u^\top,...,h_u^\top]^\top$, and define: $\Xcal_N := \{\xx\in\R^{nN}: \HH_x\xx\leq\hh_x\}$, $\Ucal_N := \{\uu\in\R^{mN}: \HH_u\uu\leq\hh_u\}$.

Then, \eqref{sbocp} can be written in function of the resulting probability distribution $\tP$ as
\begin{align}\label{scen_mpc_ini}
     \min_{\uu} \{J(x_t, \tP, \uu): \uu\in \Fcal(x_t, \tP)\},
\end{align}
where
\begin{align}
     J(x_t, \uu_t, \tP) = \sum_{j=1}^{\tM} \tpj \left(\|\xxj_t\|_1 + \|\uu_t\|_1 \right) \label{eq:comp_cost}
\end{align}
is the cost function that we minimize, and
\begin{align}\label{eq:feas}
\begin{split}
    \Fcal(x_t, \tP) := \{\uu_t\in\R^{mN}: & \sum_{j=1}^{\tM} \tpj \mathbf{1}_{\xxj_t\in\Xcal_N}\geq1-\varepsilon, \\
    & \uu_t\in\Ucal_N\}
\end{split}
\end{align}
denotes the set of feasible inputs. Note that the dependence of $J$ and $\Fcal$ on $x_t$ is implicit in $\xxj_t$ according to \eqref{eq:compact_dyn}. We also observe that the way $\tP$ is constructed affects both the cost function \eqref{eq:comp_cost} and the constraint set \eqref{eq:feas} of \eqref{scen_mpc_ini}, which coincide with \eqref{sbocp} for $\tP=\P$.

\section{Scenario reduction}\label{sec:red}

The majority of the reduction algorithms selects a set of representative scenarios in a way that some suitable similarity metric between the initial and the resulting distributions is optimized. This can be, e.g., the Wasserstein distance \cite{romisch2007stability, dupavcova2003scenario}, or a discrepancy distance \cite{henrion2008discrepancy}, between the probability distributions $\P, \tP$. Such reduction algorithms are typically of combinatorial complexity \cite{heitsch2003scenario}; hence, approximate methods to perform a reduction have been proposed in the context of stochastic programming. A popular family of reduction algorithms is based indeed on grouping similar scenarios in a clustering fashion \cite{rujeerapaiboon2022scenario}, i.e., given the initial set of scenarios $\Scal:=\{\eeta^{(1)}, ..., \eeta^{(M)}\}$ and the desired cardinality $\tM$ for the reduced scenario set, a clustering-based reduction algorithm groups similar scenarios in order to minimize an appropriate cost function, and one representative scenario is chosen for each cluster with probability equal to the sum of the probabilities of the scenarios belonging to the same cluster. In the following, we propose a novel problem-dependent clustering-based reduction algorithm that also includes the more standard $k$-means \cite{hartigan1979algorithm} and $k$-medians \cite{whelan2015understanding}. By denoting the reduced scenario set as $\tScal:=\{\teeta^{(1)}, ..., \teeta^{(\tM)}\}$, the proposed clustering-based scenario reduction algorithm minimizes the following objective function:
\begin{align*}
    \Lscr(\tScal) 
    & = \E_{\eeta\sim\P}\left[\min_{\teeta\in\tScal}\|\eeta-\teeta\|_l^l\right]
     = \sum_{h=1}^{M}\ph \min_{\teeta\in\tScal}\|\eetah-\teeta\|_l^l,
\end{align*}
where we consider $l\in\{1, 2\}$.
Let us now define a cluster as the set of indices corresponding to the closest scenarios to a given center $\teetaj$, i.e.
\begin{align*}
    \Ccalj:=\{h\in\{1,...,M\} : \arg\min_{i\in\{1,...,\tM\}}\|\eeta^{(h)}-\teeta^{(i)}\|^l_l = j\}.
\end{align*}
Hence, the objective function $\Lscr$ can be equivalently written as
\begin{align}\label{eq:clus:cost}
    \Lscr(\tScal) = \sum_{j=1}^{\tM}\sum_{h\in\Ccalj}\ph \|\eetah-\teetaj\|_l^l,
\end{align}
and the scenario reduction problem becomes
\begin{align}\label{clus}
    \min_{\tScal} \Lscr(\tScal).
\end{align}

The rationale behind the cost function \eqref{eq:clus:cost} is that we intend to minimize the expected distance between the scenarios belonging to a given cluster $\Ccalj$ and the corresponding center $\teetaj$, $j=1,...,\tM$.
The reason why in \eqref{eq:clus:cost} we explicitly consider the expectation over the distribution $\P$ will be clear in the performance analysis in the next section, where we show that there exists a connection between the loss \eqref{eq:clus:cost} and the impact of the scenario reduction on the system performance.
It is well known that \eqref{clus} is NP-hard \cite{dasgupta2008topics}; hence, greedy algorithms are frequently used to provide suboptimal solutions to \eqref{clus}. Inspired from \cite{dasgupta2008topics}, we propose the following iterative algorithm:
\begin{algorithm}[h!]
\caption{Clustering-based scenario reduction}\label{alg:red}
\begin{algorithmic}
\Require $\Scal, p^{(1)}, ..., p^{(M)}, \tM$;
\State Initialize $\tScal \subseteq \Scal$ arbitrarily;
\While{$\Lscr(\tScal)$ decreases} 
\For {$j\in\{1,...,\tM\}$}
\begin{align*}
&\begin{aligned}
    \quad 1) \  \ \Ccalj\leftarrow\{&h\in\{1,\ldots,M\} : 
\\&\teetaj=\arg\min_{\teeta^{(i)}\in\tScal}\|\eeta^{(h)}-
    \teeta^{(i)}\|^l_l\};
\end{aligned}
\\& \begin{aligned}
    \quad 2) \ \teetaj \leftarrow \arg\min_{\eeta\in\R^{Nn}}\sum_{h\in\Ccalj} \ph \|\eeta^{(h)}-\eeta\|^l_l;
\end{aligned}
\end{align*}
\EndFor
\State Assign $\tScal \leftarrow \{\teeta^{(1)}, ..., \teeta^{(\tM)}\}$;
\EndWhile
\State Assign $\tpj=\sum_{h\in\Ccalj}p^{(h)}, j=1,...,\tM$;
\State \Return $\tScal, \tp^{(1)}, ..., \tp^{(\tM)}$.
\end{algorithmic}
\end{algorithm}

Let us state the following lemma:
\begin{lemma}\label{lem:clus}
    Algorithm \ref{alg:red} converges to a (suboptimal) solution of \eqref{clus} in a finite number of iterations. In particular, if $\ph=\frac{1}{M}, h=1,...,M$, the update for $\teetaj$ in step 3) in Algorithm \ref{alg:red} corresponds to:
    \begin{itemize}
        \item[a)] $\teetaj = \sum_{h\in\Ccalj}\ph\eetah / |\Ccalj|$ if $l=2$ (i.e., $k$-means clustering);
        \item[b)] $\teetaj$ such that  $\teetaj_i$ is the median of the $i$-th entries of each $\eeta^{(h)}, h\in\Ccalj$, $i=1,...,nN$, if $l=1$ (i.e., $k$-medians clustering).
    \end{itemize}
\end{lemma}
\begin{proof}
The convergence immediately follows by observing that the updates in 1) and 2) lead to a decrease in the cost \eqref{eq:clus:cost}; hence, $\Lscr$ is monotonically decreasing in the iterations of Algorithm \ref{alg:red}. Then, a) and b) follow by observing that with $\ph=\frac{1}{M}$ Algorithm \eqref{alg:red} corresponds to $k$-means (if $l=2$) and $k$-medians (if $l=1$) clustering \cite{dasgupta2008topics}.
\end{proof}
\begin{remark}
    In the case where the probabilities are not all equal, part a) of Lemma \eqref{lem:clus} corresponds to the weighted average of the scenarios belonging to a given cluster $\Ccalj$, where the weights are $\ph, h\in\Ccalj$. Part b) corresponds to the weighted median of the scenarios belonging to a given cluster $\Ccalj$, where the weighted median is computed element-wise. For the definition and the computation of the weighted median refer to \cite{gurwitz1990weighted}.
\end{remark}

\textcolor{black}{Algorithm \ref{alg:red} can be used in an offline phase to derive a set of scenarios with reduced cardinality, which is then employed as description of the uncertainty for the online optimal control problem. However, using a reduced number of scenarios has an impact on the performance of the system, as shown in the next section.}

\section{Enforcing out-of-sample guarantees}\label{sec:oos}
The probability $\tP$ can indeed be constructed from Algorithm \ref{alg:red}, by associating to each scenario $\teetaj$ the probability $\tpj$, $j=1,...,\tM$.
The resulting distribution $\tP$ will then affect the optimal solution of problem \eqref{scen_mpc_ini}, and, moreover, considering a reduced number of scenarios in \eqref{scen_mpc_ini} might compromise the OOS guarantees of the controlled system, considering that it evolves according to the original distribution $\P$.
Hence, in this section, we focus on the effect that a scenario reduction can have on the OOS guarantees of the real system, in terms of actual chance constraint satisfaction and performance.

\subsection{Adaption of the constraint set}
Since the actual system dynamics evolve according to the distribution $\P$, it might happen that the optimal input obtained by solving \eqref{scen_mpc_ini} (i.e., the problem that employs the reduced scenario set), leads to an actual constraint violation larger than the theoretical value $\varepsilon$. Hence, in this section, we propose a sufficient condition to ensure that $\Fcal(x_t, \tP)\subseteq\Fcal(x_t, \P), \ \forall x_t\in\Xcal$. 

\begin{theorem}\label{th:feas}
    Let $\Ecal_t^{(j)}$ be a set such that $\GGamma \left(\eetah_t - \teetaj_t \right)\in\Ecal_t^{(j)}, \forall h\in\Ccalj$, and let us consider the following constraint set, obtained by tightening the state set $\Xcal_N$:
    \begin{align}\label{th:feas:tightenedset}
        \Tilde{\Fcal}(x_t, \tP) := \{\uu\in\R^{mN}: & \sum_{j=1}^{\tM} \tpj \mathbf{1}_{\xxj_t\in\Xcal_N\ominus\Ecal_t^{(j)}}\geq1-\varepsilon, \nonumber \\ & \uu_t\in\Ucal_N\}. 
    \end{align}
    Then it holds that $\Tilde{\Fcal}(x_t, \tP)\subseteq\Fcal(x_t, \P), \ \forall x_t\in\Xcal$.
\end{theorem}
\begin{proof}
Let $\teetaj, j\in\{1,...,\tM\},$ be one of the scenarios belonging to the reduced set, and let $\eeta^{(h)}$ be a scenario belonging to the same cluster, i.e.,  $h\in\Ccalj$. Let $\txxj_t$ be the scenario trajectory corresponding to scenario $\teetaj$, and $\xx^{(h)}_t$ the scenario trajectory corresponding to scenario $\eeta^{(h)}$. We now show that if $\txxj_t \in \Xcal_N\ominus\Ecal_t^{(j)}$ then $\xx^{(h)}_t \in \Xcal_N$. Indeed, for any input $\uu\in\Tilde{\Fcal}(x_t, \tP)$, we have:
\begin{align}\label{proof:belong}
\begin{split}
    \xx^{(h)}_t
    & = \FF x_t + \GG\uu_t + \GGamma \eeta^{(h)}_t 
    \\ & = \FF x_t + \GG\uu_t + \GGamma \teetaj_t + \GGamma \left(\eeta^{(h)}_t - \teetaj_t \right)    
    \\ & \in \Xcal_N\ominus\Ecal_t^{(j)}\oplus\Ecal_t^{(j)} 
    \\ & \in \Xcal_N ,
\end{split}
\end{align}
where we exploit the known property: $A\ominus B\oplus B \subseteq A$. 
Let us now prove that the relations in \eqref{proof:belong} are  satisfied with the required probability $1-\varepsilon$. The fact that $\xx^{(j)}_t \in \Xcal_N\ominus\Ecal_t^{(j)}$ implies $\xx^{(h)}_t \in \Xcal_N$, translates in $\mathbf{1}_{\xx_t^{(j)}\in(\Xcal_N\ominus\Ecal^{(j)}_t)} = 1 \Rightarrow \mathbf{1}_{\xx^{(h)}_t\in\Xcal}=1, \forall h\in\Ccalj$, or equivalently: $\mathbf{1}_{\xx^{(h)}_t\in\Xcal} \geq \mathbf{1}_{\xx_t^{(j)}\in(\Xcal_N\ominus\Ecal^{(j)}_t)}, \forall h\in\Ccalj$.  The state constraint for the scenario trajectory $\xx^{(j)}_t$ is satisfied with probability $\tpj = \sum_{h\in\Ccalj} p^{(h)}$; hence, if $\uu\in\Tilde{\Fcal}(x_t, \tP)$, it must hold that
\begin{align*}
    1-\varepsilon 
    & \leq \sum_{j=1}^{\tM} \tpj \mathbf{1}_{\xx_t^{(j)}\in(\Xcal\ominus\Ecal^{(j)}_t)} \\
    & \leq \sum_{j=1}^{\tM} \sum_{h\in\Ccalj} \ph \mathbf{1}_{\xx_t^{(h)}\in\Xcal} \\
    & \leq \sum_{j=1}^{M} \pj \mathbf{1}_{\xx_t^{(j)}\in\Xcal} \ ,
\end{align*}
and this proves that $\uu$ is feasible for the original set of scenarios, i.e.: $\uu\in\Fcal(x_t, \P)$.
\end{proof}

The meaning of Theorem \ref{th:feas} is that the optimal input sequence obtained by solving the scenario-based optimal control problem \eqref{scen_mpc_ini} with the tightened constraint set described
by \eqref{th:feas:tightenedset} is indeed feasible in terms of chance constraint satisfaction for the actual system, which evolves according to $\P$.

\subsection{Adaption of the cost function}
Even if Theorem \ref{th:feas} ensures $\Tilde{\Fcal}(x_t, \tP)\subseteq\Fcal(x_t, \P)$, it might still happen that
the solution of \eqref{scen_mpc_ini} leads to an over-optimistic expected performance of the controlled system, due to the usage of a limited number of scenarios in the cost function. Hence, by defining \begin{align}
    J^\star(x_t, \P) &:= \min_{\uu} \{J(x_t, \P, \uu): \uu\in \Fcal(x_t, \P)\}, \label{eq:optimalcost_param}    
    \\ \Tilde{J}^\star(x_t, \tP) &:= \min_{\uu} \{J(x_t, \tP, \uu): \uu\in \Tilde{\Fcal}(x_t, \tP)\}, \label{eq:optimalcost_param_red}
\end{align}
we now aim at finding a correction term $c(x_t, \uu, \P, \tP)\geq0$ such that
\begin{align*}
    \Tilde{J}^\star(x_t, \tP) + c(x_t, \uu, \P, \tP) \geq J^\star(x_t, \P), \ \forall x_t\in\Xcal.
\end{align*}

\begin{theorem}\label{th:cost}
    Let $\bc_t\in\R_{>0}$ be a constant such that
    \begin{align*}
        \bc_t \geq \sum_{j=1}^{\tM} \sum_{h\in\Ccalj} \ph \|\GGamma(\eeta_t^{(h)}-\Tilde{\eeta}_t^{(j)})\|_1 .
    \end{align*}
    Then it holds that
    \begin{align}\label{eq:th:claim}
        \Tilde{J}^\star(x_t, \tP) + \bc_t \geq J^\star(x_t, \P), \ \forall x_t\in\Xcal.
    \end{align}
\end{theorem}

\begin{proof}
    Let $\uus_t$ and $\tuus_t$ be the minimizers associated, respectively, to \eqref{eq:optimalcost_param} and \eqref{eq:optimalcost_param_red}. Then it follows that
    \begin{align}
        J^\star(x_t, \P) - \Tilde{J}^\star(x_t, \tP) & = J(x_t, \uus_t, \P) - J(x_t, \tuus_t, \tP) 
        \\ & \leq J(x_t, \tuus_t, \P) - J(x_t, \tuus_t, \tP) \label{eq:th1:subbound}
        \\ & \leq | J(x_t, \tuus_t, \P) - J(x_t, \tuus_t, \tP)|,
    \end{align}
    where in \eqref{eq:th1:subbound} we exploit that the input $\tuus_t$ is feasible for problem \eqref{sbocp} thanks to Theorem \ref{th:feas}. Then, exploiting the definition of the cost function \eqref{eq:comp_cost}, we have
    \begin{align}
        & | J(x_t, \tuus_t, \P) - J(x_t, \tuus_t, \tP)| 
        \\ & = \left|\sum_{j=1}^{M} \pj \|\xxj_t\|_1  - \sum_{j=1}^{\tM} \tpj \|\txxj_t\|_1 \right| \label{proof:cost1}
        \\& = \left|\sum_{j=1}^{\tM} \sum_{h\in\Ccalj} \ph \left( \|\xxh_t\|_1 - \|\txxj_t\|_1  \right)\right|
        \\& \leq \sum_{j=1}^{\tM} \sum_{h\in\Ccalj} \ph \|\xxh_t- \txxj_t\|_1 \label{proof:cost2}
        \\& \leq \sum_{j=1}^{\tM} \sum_{h\in\Ccalj} \ph \left\|\GGamma\left(\eeta^{(h)}_t- \Tilde{\eeta}^{(j)}_t\right)\right\|_1 , \label{proof:cost3}
    \end{align}
    where in \eqref{proof:cost1} we leverage that the applied input is the same, in \eqref{proof:cost2} we use the triangular inequality, and in \eqref{proof:cost3} we use the dynamics \eqref{eq:compact_dyn}.
    Then \eqref{eq:th:claim} follows by selecting a constant $\bc_t$ greater or equal than the quantity in \eqref{proof:cost3}.
\end{proof}

Notice in particular that the quantity $\bc_t$ does not depend on the current state $x_t$ of the system.
\begin{remark}
    From the above proof, it follows that
    \begin{align}
        J^\star(x_t, \P) &- \Tilde{J}^\star(x_t, \tP) \label{eq:perfbound}
        \\& \leq \|\GGamma\|_1\sum_{j=1}^{\tM} \sum_{h\in\Ccalj} \ph \left\|\left(\eeta^{(h)}_t- \Tilde{\eeta}^{(j)}_t\right)\right\|_1
        \\& \leq \|\GGamma\|_1\Lscr(\tScal).
    \end{align}
    Hence, the choice of the objective function in the reduction problem \eqref{clus} relates to the distance between the optimal values  \eqref{eq:perfbound} of the optimal control problem \eqref{scen_mpc_ini}. In other words, the clustering algorithm in \eqref{clus} has a problem-dependent nature, i.e., by minimizing the cost function \eqref{eq:clus:cost} we also reduce the difference in \eqref{eq:perfbound}.
\end{remark}

By putting Theorems \ref{th:feas} and \ref{th:cost} together, the following optimization problem:
\begin{align}\label{scen_final}
    \min_{\uu} \{J(x_t, \tP, \uu) + \bc_t: \uu\in \Tilde{\Fcal}(x_t, \tP)\}
\end{align}
over-approximates the original one \eqref{sbocp}. The main advantage in solving \eqref{scen_final} is that a reduced scenario set is used, which provides a consistent computational advantage. Nevertheless, the effect of the discarded scenarios is implicitly described in \eqref{scen_final} by the constraint tightening and by the over-approximation of the cost function.
Thanks to the constraint tightening, any input that is feasible for \eqref{scen_final} is also feasible for the original problem \eqref{sbocp}, and the optimal cost of \eqref{scen_final} provides an upper bound for the cost of \eqref{sbocp}. \textcolor{black}{For the constraint tightening and the over-approximation of the cost, notice that the sets $\Ecal^{(j)}_t, j=1,...,\tM,$ as well as the constant $\bc_t$, can easily be computed from the knowledge of the matrix $\GGamma$, and by measuring the distance between each scenario and the representative scenario of the same cluster to which the given scenario belongs.}

\section{Numerical example} \label{sec:ex}
Let us consider a discrete time dynamical system \eqref{eq:sys} with
\begin{align}\label{eq:example:dyn}
    A=\begin{bmatrix} 1 & 1 \\ 0 & 0.5 \end{bmatrix}, \quad B=\begin{bmatrix} 0 \\ 1 \end{bmatrix}
\end{align}
and constraints: $\Pr(x_{t+k}\geq-1, k=1,...,N-1)\geq 0.2$, $\|u_{t+k}\|_{\infty}\leq 2$. We consider a prediction horizon of $N=10$, and an initial scenario set of cardinality $M=200$. For simplicity, we assume that all the scenarios are equally likely, i.e.: $\pj=\frac{1}{M}, j=1,...,M$. Hence, Lemma \ref{lem:clus} suggests that the $k$-medians algorithm is particularly suited, since the chosen norm is the 1-norm. In the following, we provide a comparison between the solution of \eqref{scen_mpc_ini} (i.e., the problem solved considering the reduced scenario set, denoted by P1) and \eqref{scen_final} (i.e., the problem that, in addition, considers the sufficient conditions for OOS guarantees, denoted by P2) with \eqref{sbocp} (i.e., the original problem that considers the original set with $M=200$ scenarios, denoted by EXACT). For solving P1 and P2, we run both $k$-medians (denoted by kMED) and $k$-means (denoted by kMNS) with $\tM\in\{5, 25, 50, 75, 100, 125, 150, 175\}$ scenarios. We cast the resulting optimal control problem as an MILP, by associating a binary variable to each scenario, and by reformulating the chance constraint \eqref{sbocp:sc} accordingly \cite{pagnoncelli2009sample}.
For all the simulations, we use Gurobi 10.0.2 as a solver on a CPU Intel i7-1185G7 @ 3.00GHz.

In Figure \ref{fig:performance}, we compare the OOS performance of the system for both P1 and P2.
In contrast to P1, it is possible to notice that P2 offers OOS guarantees for the real system. In particular, thanks to the constraint tightening, Figure \ref{fig:performance} shows that the actual constraint violation is always lower than the theoretical value $\varepsilon$. In addition, in Figure \ref{fig:performance} we note that the optimal cost of P2 provides a more reliable estimate of the cost associated to the initial set of scenarios. In synthesis, P2 provides OOS guarantees thanks to the tightening and the over-approximation of the cost proposed in Section \ref{sec:oos}, at the cost of a mild conservatism that is gradually reduced by increasing $\tM$. We observe that both P1 and P2 converge to the same value of actual constraint violation and optimal cost as the number of scenarios is increased. We also note that $k$-medians and $k$-means lead to comparable results; however, $k$-medians, which employs the 1-norm in the cost function, slightly outperforms $k$-means. This is expected, since the norm chosen in the cost of the optimal control problem is indeed the 1-norm.
Finally, in the last plot of Figure \ref{fig:performance} we observe that the solver time required to solve P1 and P2 grows exponentially with the number of scenarios due to the mixed-integer reformulation of the chance constraint \eqref{sbocp:sc}. The solver time is comparable for P1 and P2, and P2 offers OOS guarantees even with a small number of scenarios and at a low computational cost with respect to P1.

\begin{figure}
    \hspace{-0.2cm}
    \includegraphics{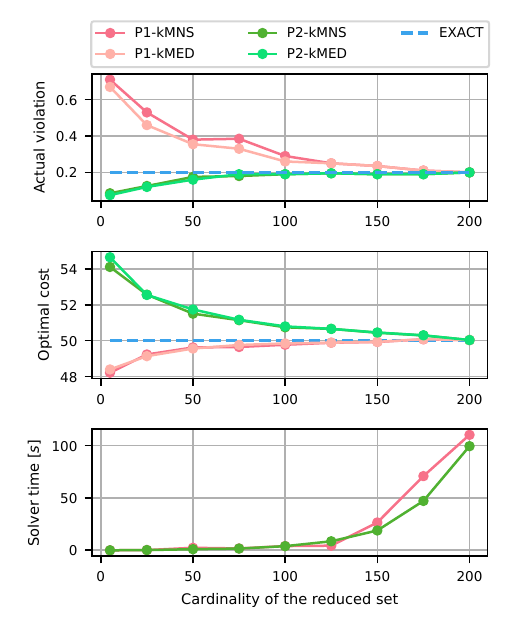}
    \caption{Performance analysis for the real system. In particular, we notice that P2 over-approximates the original problem \eqref{sbocp} both in terms of constraint satisfaction and performance, and the associated computational cost is comparable to P1.}
    \label{fig:performance}
\end{figure}

\section{CONCLUSIONS}
In this paper, we have shown that scenario reduction algorithms lead to a significant computational advantage in stochastic optimal control problems, while, however, compromising the OOS performance of the system. To tackle this issue we have proposed a scenario-reduction algorithm tailored for optimal control problems, and we have proposed sufficient conditions to implicitly keep into account the discarded scenarios in the optimal control problem at a low computational cost. The numerical example demonstrates that the proposed method leads to a consistent computational advantage, since few scenarios can be used without compromising the OOS guarantees, at the cost of a moderate conservatism.

Ongoing and future work consists in extending these results to nonlinear systems and more general sources of uncertainty, as well as providing guarantees in a closed-loop setting. Another future step is to use scenarios to approximate continuous distributions, and to provide OOS guarantees with a certain confidence level.

% \addtolength{\textheight}{-12cm}   % This command serves to balance the column lengths
                                  % on the last page of the document manually. It shortens
                                  % the textheight of the last page by a suitable amount.
                                  % This command does not take effect until the next page
                                  % so it should come on the page before the last. Make
                                  % sure that you do not shorten the textheight too much.

%%%%%%%%%%%%%%%%%%%%%%%%%%%%%%%%%%%%%%%%%%%%%%%%%%%%%%%%%%%%%%%%%%%%%%%%%%%%%%%%

%%%%%%%%%%%%%%%%%%%%%%%%%%%%%%%%%%%%%%%%%%%%%%%%%%%%%%%%%%%%%%%%%%%%%%%%%%%%%%%%

%%%%%%%%%%%%%%%%%%%%%%%%%%%%%%%%%%%%%%%%%%%%%%%%%%%%%%%%%%%%%%%%%%%%%%%%%%%%%%%%

\section*{ACKNOWLEDGMENT}
We wish to thank Matin Jafarian for the fruitful discussions and careful reading.

\bibliography{IEEEabrv,ecc}

\end{document}